\documentclass[12pt,reqno]{amsart}

\usepackage{latexsym}
\usepackage{amssymb}
\usepackage[active]{srcltx}

\newcommand{\NPHI}{\ncal_{\phi_{\lambda}}}

\usepackage{epsfig}

\newcommand{\szego}{Szeg\"o\ }

\newcommand{\R}{{\mathbb R}}
\newcommand{\C}{{\mathbb C}}

\newcommand{\dbar}{\bar\partial}
\newcommand{\ddbar}{\partial\dbar}

\newcommand{\half}{{\frac{1}{2}}}

\renewcommand{\phi}{\varphi}

\newcommand{\acal}{\mathcal{A}}

\newcommand{\gcal}{\mathcal{G}}
\newcommand{\hcal}{\mathcal{H}}
\newcommand{\ical}{\mathcal{I}}

\newcommand{\lcal}{\mathcal{L}}

\newcommand{\ncal}{\mathcal{N}}

\newcommand{\ocal}{\mathcal{O}}

\newcommand{\scal}{\mathcal{S}}

\newcommand{\la}{\lambda}

\newtheorem{theo}{{\sc Theorem}}
\newtheorem{cor}[theo]{{\sc Corollary}}

\newtheorem{lem}[theo]{{\sc Lemma}}
\newtheorem{prop}[theo]{{\sc Proposition}}

\newenvironment{rem}{\medskip\noindent{\it Remark:\/} }{\medskip}

\title[ Analytic Steklov eigenfunctions]{Measure of nodal sets of analytic Steklov eigenfunctions}

\author{Steve Zelditch}

\address{Department of Mathematics, Northwestern  University,
Evanston, IL 60208-2370, USA} \email{
zelditch@math.northwestern.edu}

\thanks{Research  partially supported by NSF grant   DMS-1206527.}

\begin{document}
\maketitle

\begin{abstract}
 Let $(\Omega, g)$ be a real analytic Riemannian manifold with real analytic boundary $\partial \Omega$.
Let $\psi_{\lambda}$ be an  eigenfunction of the Dirichlet-to-Neumann operator $\Lambda$ of $(\Omega, g, \partial \Omega)$ of eigenvalue $\lambda$. 
Let $\ncal_{\lambda_j}$ be its nodal set. Then, there exists a constant $C > 0$ depending only on
 $(M, g, \partial \Omega)$ so that
$$\hcal^{n-2} (\ncal_{\lambda}) \leq C \lambda.$$
This proves a conjecture of F. H. Lin and K. Bellova.
\end{abstract}

This article is concerned with the  Hausdorff $\hcal^{n-2}$ (surface) measure of the nodal sets
$$\ncal_{\lambda} = \{x \in \partial \Omega: \psi_{\lambda} (x) = 0\} \subset \partial \Omega$$
of Steklov eigenfunctions of eigenvalue $\lambda$ of a domain $\Omega \subset \R^n$ in the real analytic case.
The Steklov eigenvalue problem is to find the eigenfunctions
of the   Steklov problem on a domain $\Omega$,
\begin{equation} \label{SP} \left\{ \begin{array}{l} \Delta  u(x) = 0, \;\;  x \in \Omega, \\ \\
\frac{\partial u}{\partial \nu}(x) = - \lambda u(x),  \;\; x \in \partial \Omega. \end{array} \right. \end{equation}
It is often assumed that $\Omega \subset \R^n$ is a bounded $C^2$ domain with  Euclidean
metric, but the problem may be posted on a bounded domain in any Riemannian manifold. 
The eigenvalue problem may be reduced to the boundary, and $\psi_{\lambda}$ is an eigenfunction
\begin{equation} \Lambda \psi_{\lambda} = \lambda \psi_{\lambda} \end{equation} of  the Dirichlet-to-Neumann operator $$\Lambda f = \frac{\partial u}{\partial \nu}(x)  |_{\partial \Omega}. \;\;$$
Here, $u$ is the harmonic extension of $f$,  $$\left\{ \begin{array}{l} \Delta  u(x) = 0, \;\;  x \in \Omega, \\ \\
u(x)=f(x),  \;\; x \in \partial \Omega. \end{array} \right..$$
 $\Lambda$ is self-adjoint on $L^2(\partial \Omega, d S)$  and there exists  an orthonormal
basis $\{\psi_j\}$ of eigenfunctions $$\Lambda \psi_j = \lambda_j \psi_j, \;\;\; \psi_j \in C^{\infty}(\partial \Omega), \;\;
\int_{\partial \Omega} \psi_j \psi_k d S = \delta_{jk},$$
where $d S$ is the surface measure. We  order the eigenvalues in ascending order  $0=\la_0<\la_1\le \la_2\le\cdots$,
counted with multiplicty.

In a recent article, Bellova-Lin \cite{BL} proved that 
$$\hcal^{n-2} (\ncal_{\lambda}) \leq C \lambda^{6}$$
when $\Omega\subset \R^n$ is a bounded  {\it real analytic} Euclidean domain.  They suggest   that the optimal result is
$\hcal^{n-2} (\ncal_{\lambda}) \leq C \lambda.$ The purpose of this article is to prove this upper bound for bounded  real analytic domains in general real analytic Riemannian manifolds.

\begin{theo} \label{NODALBOUND} Let $(\Omega, g)$ be a real analytic Riemannian manifold with real analytic boundary $\partial \Omega$.
Let $\psi_{\lambda}$ be an  eigenfunctions of the Dirichlet-to-Neumann operator $\Lambda$ of $(\Omega, g, \partial \Omega)$ of eigenvalue $\lambda$, and
 $\ncal_{\lambda}$ be its nodal set as above. Then, there exists a constant $C > 0$ depending only on
 $(\Omega, g, \partial \Omega)$ so that
$$\hcal^{n-2} (\ncal_{\lambda}) \leq C \lambda.$$

\end{theo}
It is not hard to find examples of  $(\Omega, g, \partial \Omega)$ and $\psi_{\lambda}$ where the upper
bound is achieved, for instance  on a
 hemisphere of a round sphere. But it is not clear that it is attained by a sequence of Steklov eigenfunctions on
every  $(\Omega, g, \partial \Omega)$, or more stringently that it is obtained by every sequence of eigenfunctions.   In the setting  of real analytic Riemannian manifolds $(M,g)$, it is proved in \cite{DF}
that there exists $ C> 0$ depending only on the metric $g$ so that  $\hcal^{n-1}(\ncal_{\lambda}) \geq C \lambda$.
Since $\dim \partial \Omega = n-1$, the   analogous lower bound  for the real analytic  Steklov problem   would 
be  $\hcal^{n-2} (\ncal_{\lambda}) \geq C \lambda$, where $C$ depends only  on  $(\Omega, g, \partial \Omega)$.
However the key existence result for $\Delta$-eigenfunctions of eigenvalue $\lambda^2$, that every ball of radius $\frac{C}{\lambda}$ contains
a zero of $\phi_{\lambda}$, does not seem to be known for the Steklov problem \eqref{SP}. We believe it is possible
to prove good lower bounds by the methods of this article, and plan to investigate lower bounds in a subsequent article.





\subsection{Outline of the proof of Theorem \ref{NODALBOUND}}

The key to proving the sharp upper bound in the generality of Theorem \ref{NODALBOUND} is to use the wave group 
\begin{equation} \label{WG} U(t ) = e^{it \Lambda} : L^2(\partial \Omega) \to L^2(\partial \Omega) \end{equation}
generated by $\Lambda$. $\Lambda$ is a positive elliptic self-adjoint pseudo-differential operator of order
one, and its wave group has been constructed as a Fourier integral operator in \cite{Hor,DG}. As in \cite{Z}
we study nodal sets by analytically continuing the Schwartz kernel of the  wave group to imaginary time $t + i \tau$ with
$ \tau > 0$, and to the complexification $(\partial \Omega)_{\C} $ of $\partial \Omega$. The analytic continuation in time and in the
first space variable defines  the Poisson wave kernel\begin{equation} \label{PWG} U_{\C}(t + i \tau, \zeta, y) = e^{i (t + i \tau) \Lambda} (\zeta, y) : L^2(\partial \Omega) \to L^2((\partial \Omega)_{\C}). \end{equation}
As discussed below,  
 $\Lambda$ is  an analytic pseudo-differential operator on $\partial \Omega$ when $(\Omega, \partial \Omega, g)$ is real analytic, and \eqref{PWG}   is a   Fourier integral operator with complex phase. (See \cite{Bou2,Sj} for background
on analytic pseudo-differential operators).

In the real analytic  case, the Steklov eigenfunctions are real analytic on $\partial \Omega$ and have complex analytic extensions to
$(\partial \Omega)_{\C}$.   We then study their complex nodal sets
\begin{equation} \label{CXN} \ncal_{\lambda}^{\C} = \{ \zeta \in (\partial \Omega)_{\C}: \psi^{\C}_{\lambda_j}(\zeta) = 0\}.
\end{equation}
To prove Theorem \ref{NODALBOUND}, we  use Crofton's formula and a multi-dimensional Jensen's formula  to give an upper bound for $\hcal^{n-2}(\ncal_{\lambda})$ in terms of the integral
geometry of $\ncal_{\lambda}^{\C}$. The integral geometric approach to the  upper bound is inspired by  the classic paper of
Donnelly-Fefferman \cite{DG} (see also \cite{Lin}). But,  instead of doubling estimates or frequency function
estimates,  we use the Poisson wave kernel
to obtain growth estimates on eigenfunctions, and then use results on pluri-subharmonic functions rather than functions
of one complex variable to relate growth of zeros to growth of eigenfunctions.  This approach was used in \cite{Z}
to prove equidistribution theorems for complex nodal sets when the geodesic flow is ergodic. The Poisson wave kernel
approach   works for Steklov eigenfunctions as well as Laplace eigenfunctions,
and in fact for eigenfunctions of any positive elliptic analytic pseudo-differential operator.

We first  use the Poisson wave group  \eqref{PWG}  to analytically continue eigenfunctions in the 
form
\begin{equation} \label{UAC} U_{\C}(i \tau) \psi_j (\zeta) = e^{- \tau \lambda_j} \psi_j^{\C} (\zeta). \end{equation}
We then use \eqref{UAC} to determine the growth properties of $\psi_j^{\C}(\zeta)$ in Grauert tubes of 
the complexification of $\partial \Omega$.  
The relevant notion of Grauert tube  is the standard Grauert tube for $\partial \Omega$ with the metric $g_{\partial \Omega}$
induced by the ambient metric $g$ on $M$. This is  because the principal
symbol of $\Lambda$ is the same as the principal symbol of  $\sqrt{\Delta_{\partial \Omega}}$.

\begin{rem} A remark on notation: In \cite{Z} we use $M$ to denote a Riemannian manifold, $M_{\epsilon}$ its Grauert tube
of radius $\epsilon$ and $\partial M_{\epsilon}$ to denote the boundary of the Grauert tube of radius $\epsilon$. Since
$\partial \Omega$ is the Riemannian manifold of interest here, we denote it by $M$,   \begin{equation} \label{MOM} (M, g): = (\partial \Omega, g_{\partial \Omega}). \end{equation}   Thus the   Grauert tube of radius
$\tau$ of $(\partial \Omega)_{\C}$ is denoted $M_{\tau}$ and its boundary
by $\partial M_{\tau}$, not to be confused with $\partial \Omega$. We also denote $m = \dim M  = n -1$. 

\end{rem}
Because $U_{\C}(i \tau) $ is a Fourier integral operator with complex phase, it can only magnify the $L^2$ norm
of $\psi_j$ by a power of $\lambda_j$. Hence the exponential  $e^{\tau \lambda_j}$ dominates the $L^2$ norm
on the boundary of the Grauert tube of radius $\tau$. We prove:
\begin{prop} \label{PW} Suppose  $(\Omega, g, \partial \Omega)$ is real analytic. Let  $\{\psi_{\lambda}\}$ be an 
eigenfunction of $\Lambda$ on $M = \partial \Omega$ of eigenvalue $\lambda$.  Then
$$ \sup_{\zeta \in M_{\tau}} |\psi^{\C}_{\lambda}(\zeta)| \leq C
  \lambda^{\frac{m+1}{2}} e^{\tau \lambda},
$$
\end{prop}  
The proof follows  from  a standard cosine Tauberian result and the fact that the complexified
Poisson kernel is a complex Fourier integral operator of finite
order.  This simple growth estimate replaces the doubling estimates
of \cite{DF} and \cite{BL}. 
It is closely related to growth estimates of $\Delta$-eigenfunctions in \cite{Z,Z2,Z3}.

 For the precise   statement that $U_{\C}(t + i \tau)$ is indeed a Fourier
integral operator with complex phase, we refer to Theorem \ref{BOUFIO}. 
It is in some sense a known result for elliptic analytic pseudo-differential operators, and we therefore postpone  the detailed  proof of Theorem \ref{BOUFIO} for $\Lambda$  to a later article.

We thank Boris Hanin, Peng Zhou, Iosif Polterovich, Chris Sogge  and particularly Y. Canzani for comments/corrections on earlier versions. We also thank  G. Lebeau for confirming that
Theorem \ref{BOUFIO} should be true, with not too different a proof than in the Riemannian wave case.

\section{Geometry and analysis of Grauert tubes}

We briefly review the geometry and analysis on Grauert tubes of real analytic Riemannian manifolds. We refer to \cite{Z,Z2,GS1,GS2} for more detailed discussions.

\subsection{\label{AC} Analytic continuation  to a Grauert tube}

A real analytic manifold $M$ always possesses a complexification
$M_{\C}$, i.e. a complex manifold  of which $M$ is a totally real
submanifold. A real analytic Riemannian metric  $g$ on $M$ determines a
canonical plurisubharmonic function $\rho_g$ on $M_{\C}$; since the metric is fixed througout, we denote
it simply by $\rho$.  Its
square-root $\sqrt{\rho}$ is known as the Grauert tube function; it
 equals
$\sqrt{- r^2_{\C}(z, \bar{z})}/2$ where $r_{\C}$ is the holomorphic extension
of the distance function. The $(1,1)$ form 
$\omega = \omega_{\rho}: = i \ddbar \rho$ defines  a K\"ahler metric on $M_{\C}$.    The
Grauert tubes $M_{\epsilon}: = \{ z \in M_{\C}: \sqrt{\rho}(z) < \epsilon\} $
are strictly pseudo-convex domains in $M_{\C}$, whose boundaries are strictly
pseudo-convex CR manifolds. We also denote the contact form of $\partial M_{\tau}$ by
\begin{equation} \label{alpha} \alpha = \frac{1}{i} \partial \rho|_{\partial M_{\tau} } = d^c \sqrt{\rho}.\end{equation}

The complexified exponential map 
\begin{equation} \label{E} (x, \xi) \in B_{\epsilon}^*M \to E(x, \xi): = \exp_x^{\C} \sqrt{-1} \xi \in M_{\epsilon} \end{equation}
defines a 
symplectic diffeomorphism, where $B^*_{\epsilon} M
\subset T^*M$ is the co-ball bundle of radius $\epsilon$, equipped with   the standard symplectic structure,
and where $M_{\epsilon}$ is equipped with  $\omega_{\rho}$. The Grauert tube function $\sqrt{\rho}$ pulls back under $E$ to the metric norm function $|\xi|_g$. 
We emphase the setting $M_{\C}$ but it is equivalent to using $E$ to endow $B^*_{\epsilon} M$ with an adapted
complex structure. We refer to  \cite{GS1, GS2, LS, GLS} for further discussion.

\subsection{Geodesic and Hamiltonian flows}

The microlocal analysis of the kernels \eqref{PWG} 
involves the  complexification of the geodesic flow.
We denote by $g^t$ 
the  (real) homogeneous geodesic flow of  $(M, g)$. It is   the real analytic
Hamiltonian flow on $T^*M \backslash 0_M$ generated by the Hamiltonian $|\xi|_g$
with respect to the standard symplectic form $\omega$. We also consider
 the Hamiltonian flow of
$|\xi|_g^2$, which is real analytic on all of $T^*M$ and  denote
its Hamiltonian flow by $G^t$. In general, we denote by $\Xi_H$
the Hamiltonian vector field of a Hamiltonian $H$ and its flow by
$\exp t \Xi_H$.  Thus, we consider the  Hamiltonian flows
\begin{equation} \label{gtdef} 
g^t = \exp t \Xi_{|\xi|_g}, \;\;\; \mbox{resp.}\;\;\;
G^t = \exp t \Xi_{|\xi|_g^2}. \end{equation}
 The
exponential map is the map $\exp_x: T^*M \to M$ defined by $\exp_x
\xi = \pi G^t(x, \xi)$ where $\pi$ is the standard projection.
Since  $E^* \sqrt{\rho} = |\xi|, $  $E^*$ conjugates the geodesic flow
on $B^*M$ to the Hamiltonian flow $\exp t \Xi_{\sqrt{\rho}}$ of
$\sqrt{\rho}$ with respect to $\omega$, i.e.
\begin{equation} \label{gt} E(g^t(x, \xi)) = \exp t \Xi_{\sqrt{\rho}} (\exp_x i \xi).  \end{equation}

\subsection{Szeg\"o kernel and analytic continuation of the Poisson kernel}

  We denote by $\ocal^{s +
\frac{m-1}{4}}(\partial M _{\tau})$ the Sobolev spaces of CR
holomorphic functions on the boundaries of the strictly
pseudo-convex domains $M_{\tau}$, i.e.
\begin{equation} \label{SOBSP} {\mathcal O}^{s +
\frac{m-1}{4}}(\partial M_{\tau}) = W^{s + \frac{m-1}{4}}(\partial
M_{\tau}) \cap \ocal (\partial M_{\tau}), \end{equation}  where
$W^s$ is the $s$th Sobolev space and where $ \ocal (\partial
M_{\tau})$ is the space of boundary values of holomorphic
functions. The inner product on $\ocal^0 (\partial M _{\tau} )$ is
with respect to the Liouville measure or contact volume form

\begin{equation}  \label{CONTACTVOL} d\mu_{\tau} : = \alpha \wedge \omega^{m-1}, \end{equation}
on $\partial M_{\tau}$.


The study of norms of complexified eigenfunctions is 
related to the study of the \szego kernels $\Pi_{\tau}$ of
$M_{\tau}$, namely the orthogonal projections

\begin{equation} \Pi_{\tau}: L^2(\partial M_{\tau}, d\mu_{\tau}) \to \ocal^0(\partial M_{\tau},
d\mu_{\tau}) \end{equation}  onto the Hardy space of boundary
values of holomorphic functions in $M_{\tau}$ which belong to $
L^2(\partial M_{\tau}, d\mu_{\tau})$. The \szego projector
$\Pi_{\tau}$ is a complex Fourier integral operator with a
positive complex canonical relation. The
real points of its canonical relation form the graph
$\Delta_{\Sigma}$ of the identity map on the symplectic one
 $\Sigma_{\tau}
\subset T^*
\partial M_{\tau}$ defined by the spray \begin{equation} \label{SIGMATAU} \Sigma_{\tau} =
\{(\zeta, r d^c \sqrt{\rho}(\zeta)): r \in \R_+\} \subset T^*
(\partial M_{\tau})
\end{equation}  of the contact form $d^c \sqrt{\rho}$. There exists a symplectic equivalence  (cf.
\cite{GS2})
\begin{equation} \iota_{\tau} : T^*M - 0 \to
\Sigma_{\tau},\;\; \iota_{\tau} (x, \xi) = (E(x, \tau
\frac{\xi}{|\xi|}), |\xi|d^c \sqrt{\rho}_{E(x, \tau
\frac{\xi}{|\xi|})} ).
\end{equation}

\subsection{Analytic continuation of the Poisson wave kernel}

The wave group generated by $\Lambda$ on  $M = \partial \Omega$  is the unitary group $U(t) = e^{ i
 t \Lambda}$. Its kernel $U(t, x, y)$ solves the  `half-wave equation',
\begin{equation} \label{HALFWE} \left(\frac{1}{i} \frac{\partial }{\partial t} -\Lambda_x \right) U(t, x, y) = 0, \;\; U(0, x, y) =
\delta_y(x). \end{equation}  Here, $\Lambda_x$ means that $\Lambda$ is applied in the $x$ variable.  In the real domain it is   well known \cite{Hor,DG} that
$U(t, x, y)$ is the Schwartz kernel of a Fourier integral
operator,
$$U(t, x, y) \in I^{-1/4}(\R \times M \times M, \Gamma)$$
with underlying canonical relation $$\Gamma = \{(t, \tau, x, \xi,
y, \eta): \tau + |\xi| = 0, g^t(x, \xi) = (y, \eta) \} \subset T^*
\R \times T^*M \times T^*M. $$

The Poisson-wave kernel is the analytic continuation $U(t + i \tau, x, y)$  of the wave kernel  with respect to time,  $ t \to t + i \tau\in \R \times \R_+$. For $t = 0$ and  for $\tau > 0$, we
obtain
the  Poisson semi-group
$U(i \tau) = e^{- \tau \Lambda}$. For general $t + i \tau$,  the Poisson-wave kernel   has the eigenfunction expansion,
\begin{equation}\label{POISEIGEXP}  U ( i
\tau, x, y) = \sum_j e^{i (t + i \tau) \lambda_j} \psi_{\lambda_j}(x)
\psi_{\lambda_j}(y).
\end{equation}

As stated in Theorem \ref{BOUFIO} in the introduction, the  Poisson-wave  kernel $U(t + i \tau, x, y)$  admits an analytic
continuation $U_{\C}(t + i \tau, \zeta, y)$ in the first variable
to  $M_{\tau} \times M$.

\begin{theo}\label{BOUFIO}  Let $U(t)$ be the wave group of the Dirichlet to Neumann operator $\Lambda$ on
$M = \partial \Omega$ as above. Then $\Pi_{\epsilon} \circ U (i \epsilon): L^2(M)
\to \ocal(\partial M_{\epsilon})$ is a  complex Fourier integral
operator of order $- \frac{m-1}{4}$  associated to the canonical
relation
$$\Gamma = \{(y, \eta, \iota_{\epsilon} (y, \eta) \} \subset T^* \partial M_{\epsilon} \times \Sigma_{\epsilon}.$$
Moreover, for any $s$,
$$\Pi_{\epsilon} \circ U (i \epsilon): W^s(M) \to {\mathcal O}^{s +
\frac{m-1}{4}}(\partial M_{\epsilon})$$ is a continuous
isomorphism.
\end{theo}  
This statement is asserted by Boutet de Monvel in \cite{Bou, Bou2} for any real analytic positive  elliptic pseudo-differential operator,
and has been accepted since then as an established fact (see for instance \cite{GS1,GS2}).  The proof was only  sketched in \cite{Bou,Bou2},
 and the first  complete proofs appeared only recently in
\cite{Z2,L,St} for the special case of the wave group of a Riemannian manifold without boundary. 
Roughly the same proof  applies to the Steklov problem as well because $\sqrt{\Delta_{\partial M}}$
and $\Lambda$ are the same to leading order and in fact differ by an analytic pseudo-differential operator
of order zero.   This is because  the  principal symbol of $\Lambda,$
\begin{equation} \label{ps} \sigma_{\Lambda} : T^* \partial \Omega \to \R, \;\;\; \sigma_{\Lambda} (x, \xi) = |\xi|_{g_{\partial}},
\end{equation}
is the same as for the Laplacian $\Delta_{\partial}$ of the boundary $(\partial \Omega, g_{\partial})$. In fact, the 
complete symbol of $\Lambda$ is calculated in \cite{LU} (see also \cite{PS}).  It would be desirable to have
a detailed exposition of the proof, but we postpone that to a future article.

\section{Growth of complexified eigenfunctions proof of Proposition \ref{PW}}

We further need to generalize sup norm estimates
of complexified eigenfunctions in  \cite{Z2} to the $\Lambda$-eigenfunctions.

 As in \cite{Z2,Z3} we prove Proposition \ref{PW} by introducing the 
`tempered'  spectral
projections
\begin{equation}\label{TCXSPM}   P_{ I_{\lambda}}^{\tau}(\zeta, \bar{\zeta}) =
\sum_{j: \lambda_j \in  I_{\lambda}} e^{-2 \tau \lambda_j}
|\psi_{\lambda_j}^{\C}(\zeta)|^2, \;\; (\sqrt{\rho}(\zeta) \leq
\tau),
\end{equation}
where $I_{\lambda} $ could be a short interval  $[\lambda, \lambda
+ 1]$ of frequencies or a long window $[0, \lambda]$.
Exactly as in \cite{Z2} but with the wave group of $\Lambda$ replacing the wave group of $\sqrt{\Delta}$, we prove
\begin{equation} \label{PTAU}  P^{\tau}_{[0, \lambda]}(\zeta, \bar{\zeta}) =  (2\pi)^{-m} \left(\frac{\lambda}{\sqrt{\rho}} \right)^{\frac{m-1}{2}}
  \left( \frac{\lambda}{(m-1)/2 + 1} +  O (1) \right), \;\; \zeta \in \partial M_{\tau}. \end{equation}
We then obtain

\begin{cor} \label{PWa}
Let $\psi_{\lambda}$ be an eigenfunction of $\Lambda$ as above. Then 
there exists $C > 0$ so that for all $\sqrt{\rho}(\zeta) = \tau$,
$$ C
\lambda_j^{-\frac{m-1}{2}} e^{ \tau \lambda} \leq \sup_{\zeta \in
M_{\tau}} |\psi^{\C}_{\lambda}(\zeta)| \leq C
  \lambda^{\frac{m-1}{4} + \half} e^{\tau \lambda}.
$$

\end{cor}

The lower bound is not used in the nodal analysis.

\subsection{Proof of the local Weyl law}


We only sketch the proof for the sake of completeness, since it is essentially the same as in \cite{Z,Z2,Z3} and 
closely follows \cite{DG}. The novelty is that we apply the argument of \cite{DG} to the analytically continued
parametrix.

By \cite{Hor, DG} the positive elliptic first order pseudo-differential operator  $\Lambda$ generates 
a wave group which has a parametrix of the form,
\begin{equation} \label{PARAONE} U(t, x, y) = \int_{T^*_y M} e^{ i
t |\xi|_{g_y} } e^{i \langle \xi, \exp_y^{-1} (x) \rangle} A(t, x,
y, \xi) d\xi
\end{equation} 
similar to that of the wave kernel of $M = \partial \Omega$,
since $\Lambda = \sqrt{\Delta_M} + Q$ where $Q $ is an analytic pseudo-differential operator of order zero. 
Here,
 $|\xi|_{g_x} $ is the metric norm function at
$x$, and where $A(t, x, y, \xi)$ is a polyhomogeneous amplitude of
order $0$ which is supported near the diagonal. The amplitude is different from that of the wave kernel since
the transport equations involve $Q$. 

By Theorem \ref{BOUFIO}, the wave group and parametrix may be analytically continued. 
To obtain uniform asymptotics,  we use the analytic continuation of the
H\"ormander 
parametrix (\ref{PARAONE}). We choose local coordinates near $x$
and write $\exp_x^{-1}(y) = \Phi(x, y)$ in these local coordinates
for $y$ near $x$, and write the integral $T^*_yM$ as an integral
over $\R^m$ in these coordinates. The holomorphic extension of the
parametrix  to the Grauert tube $|\zeta| < \tau$ at time $t + 2 i
\tau$ has the form
\begin{equation} \label{CXPARAONE} U_{\C}(t + 2 i \tau,
\zeta, \bar{\zeta}) = \int_{\R^m} e^{(i t - 2\tau )  |\xi|_{g_y} }
e^{i \langle \xi, \Phi (\zeta, \bar{\zeta}) \rangle} A(t, \zeta,
\bar{\zeta}, \xi) d\xi,
\end{equation}
where $A$ is the  analytic extensions of the real analytic $A$ and $\Phi(\zeta, \bar{\zeta})$ is the 
analytic extension of $\exp_y^{-1} (x)$.

 We 
introduce  a cutoff function $\chi \in \scal(\R)$ with $\hat{\chi}
\in C_0^{\infty}$ supported in sufficiently small neighborhood of
$0$ so that no other singularities of $U_{\C}(t + 2 i \tau, \zeta,
\bar{\zeta})$ lie in its support. We also assume $\hat{\chi}
\equiv 1$ in a smaller neighborhood of $0$.  We then change
variables $\xi \to \lambda \xi$ and  apply the complex
stationary phase to the integral,
\begin{equation} \label{CXPARAONEc}\begin{array}{l}  \int_{\R} \hat{\chi}(t) e^{-i \lambda t} U_{\C} (t + 2 i \tau,
\zeta, \bar{\zeta})dt \\ \\
= \lambda^m \int_{0}^{\infty} \int_{\R} \hat{\chi}(t) e^{-i
\lambda t} \int_{S^{m-1}} e^{(i t - 2\tau ) \lambda r} e^{i r
\lambda \langle \omega, \Phi (\zeta, \bar{\zeta}) \rangle} A(t,
\zeta, \bar{\zeta}, \lambda r \omega ) r^{m-1} dr dt
d\omega.\end{array}
\end{equation}
The resulting integral \eqref{CXPARAONEc} is a semi-classical Fourier integral distribution of with a complex phase,
the same phase  as in the pure Riemannian case treated in \cite{Z2}. Hence the stationary phase calculation 
is essentially the same as in    section 9.1
of \cite{Z2}.  We first integrate over $d \omega$ and find that there are two stationary
phase points, one giving an exponentially decaying amplitude of order $e^{- 2 \lambda \tau r}$and one for which the critical value is $2 \lambda  \tau r$. It cancels the term $- 2 \tau \lambda r$ coming from  the factor $e^{(i t - 2\tau ) \lambda r} $. 
We then apply stationary phase to the resulting 
integral over $(t, r)$
with phase $ t (r - 1)$. The critical set consists of $r = 1, t = 0$. The
phase is clearly  non-degenerate with Hessian determinant one and inverse
Hessian operator $D^2_{\theta, t}$. Taking into account the factor of $\lambda^{-1}$ from the change of variables,
the stationary phase expansion gives
\begin{equation}\label{EXPANSIONCaa}  \sum_j \psi(\lambda - \lambda_j) e^{- 2 \tau
\lambda_j} |\psi_j^{\C}(\zeta)|^2 \sim \sum_{k = 0}^{\infty}\lambda^{\frac{m-1}{2} - k} \omega_k(\tau; \zeta),
\end{equation}
where the coefficients $\omega_k(\tau, \zeta)$ are smooth for $\zeta \in \partial M_{\tau}$.
The Weyl asymptotics  then follows from the standard cosine Tauberian theorem, as in \cite{DG} or \cite{Z2} (loc. cit.).



\section{Proof of Theorm \ref{NODALBOUND}}

We start with the integral geometric approach of   \cite{DF} (Lemma 6.3) (see also  \cite{Lin}  (3.21)). There exists a ``Crofton formula" in the  real domain which
bounds  the local   nodal
hypersurface volume above,
\begin{equation} \label{INTGEOM} \hcal^{m-1}(\NPHI \cap U)  \leq C_L \int_{\lcal} \#\{ \NPHI\cap \ell\}
d\mu(\ell). \end{equation}
 Thus, $ \hcal^{m-1}(\NPHI \cap U)  $ is bounded above
by a constant $C_L$ times the average  over all line segments of length $L$  in a local coordinate patch $U$  of the number of intersection points
of the line with the nodal hypersurface. The measure $d\mu_L$ is known as the `kinematic measure' in the Euclidean
setting \cite{F} (Chapter 3); see also Theorem 5.5 of \cite{AP}.  We will be using geodesic
segments of fixed length $L$ rather than line segments,  and parametrize them by $S^*M \times [0, L]$, i.e. by
their intial data and time.  Then
$d\mu_{\ell}$ is essentially Liouville measure $d\mu_L$ on $S^* M$ times $dt$.

The complexification of a
real line $\ell = x + \R v$ with $x, v \in \R^m$ is $\ell_{\C} = x +
\C v$. Since the number of intersection points (or zeros) only increases if we count complex intersections, we have
\begin{equation} \label{INEQ1} \int_{\lcal} \# (\NPHI \cap \ell)
d\mu(\ell) \leq \int_{\lcal} \# (\NPHI^{\C} \cap \ell_{\C})
d\mu(\ell).
 \end{equation} 
Note that this complexification is quite different from using intersections with all complex lines to measure
complex nodal volumes. If we did that, we would obtain a similar  upper bound on the complex hypersurface volume
of the complex nodal set. But it would not give an upper bound on the real nodal volume and indeed would
the complex volume tends to zero as one shrinks the Grauert tube radius to zero, while \eqref{INEQ1} stays
bounded below. 

Hence to prove Theorem \ref{NODALBOUND} it suffices to show
\begin{lem} \label{DF2} We have,
$$\hcal^{m-1}(\NPHI) \leq C_L \int_{\lcal} \# (\NPHI)^{\C} \cap \ell_{\C} )
d\mu(\ell) \leq C \lambda. $$

\end{lem}

We now sketch the proofs of these results using a somewhat novel approach to the integral geometry and
complex analysis.

\subsection{\label{GEOS} Background on hypersurfaces and geodesics}

The proof of the  Crofton formula given below in Lemma \ref{CROFTONEST} involves the geometry
of geodesics and hypersurfaces. To prepare for it we provide the relevant background.

As above, we  denote by $d\mu_L$  the Liouville measure on $S^*
M$. We also denote by $\omega$ the standard symplectic form on
$T^* M$ and by $\alpha$ the canonical one form. Then 
$d\mu_L = \omega^{n-1} \wedge \alpha$ on $S^* M$. Indeed, $d\mu_L$
is characterized by the formula $d\mu_L \wedge d H = \omega^{m}$,
where $H(x, \xi) = |\xi|_g$.  So it suffices to verify that
$\alpha \wedge dH = \omega$ on $S^*M$. We take the interior
product $\iota_{\Xi_H}$ with the Hamilton vector field $\Xi_H$ on
both sides, and the identity follows from the fact that
$\alpha(\Xi_H) = \sum_j \xi_j \frac{\partial H}{\partial \xi_j} =
H = 1$ on $S^*M$, since $H$ is homogeneous of degree one. Henceforth we denote
by $\Xi = \Xi_H$ the generator of the geodesic flow. 

Let $N \subset M$ be a smooth hypersurface in a Riemannian manifold $(M,
g)$.  We denote by  $T^*_N M$ the 
of covectors with footpoint on $N$ and $S^*_N M$ the unit covectors along $N$.
We   introduce  Fermi normal coordinates $(s, y_n) $
along
 $N$,  where  $s$ are coordinates on $N$ and $y_n$  is the
 normal coordinate, so that $y_m = 0$ is a local defining function for $N$.   We also let $\sigma, \xi_m$ be
 the dual symplectic Darboux coordinates. Thus the canonical
 symplectic form is $\omega_{T^* M } = ds \wedge d \sigma + dy_m
 \wedge d \xi_m. $
Let $\pi: T^* M \to M$ be the natural projection. For notational simplicity we denote
$\pi^*y_m$ by  $ y_m$ as functions on $T^* M$. Then $y_m$ is a
defining function of  $T^*_N M$.

  The hypersurface  $S^*_N M \subset S^* M$ is a kind of Poincar\'e section or
symplectic transversal to the  orbits of $G^t$, i.e. is a symplectic transversal away from 
the (at most  codimension one) set  of $(y, \eta) \in S_N^* M$  for which 
$\Xi_{y, \eta} \in T_{y, \eta} S^*_N M$, where as above $\Xi$ is the  generator 
of the geodesic flow.  More precisely,

\begin{lem} \label{NSYMP} The restriction $\omega |_{S_N^* M}$ is symplectic on $S^*_N M \backslash S^* N$. 
\end{lem}

Indeed, $\omega |_{S_N^* M}$ is symplectic on $T_{y, \eta} S^* N$ as long as $T_{y, \eta} S^*_N M$ 
is transverse to $\Xi_{y, \eta}$, since $\ker (\omega|_{S^*M}) = \R \Xi. $ But  $S^* N$ is the set of points of $S^*_N M$ where $\Xi \in T S^*_N M$, i.e. where $S^*_N M$
fails to be transverse to $G^t$. 
Indeed, transversality fails when $\Xi(y_m) =dy_m (\Xi) = 0$, and $\ker d y_m \cap \ker d H  = T S^*_N M$.  One may also see
it in Riemannian terms as follows:  the generator $\Xi_{y, \eta}$ is the
horizontal lift $\eta^h$ of $\eta$ to $(y, \eta)$ with respect to
the Riemannian connection on $S^* M$, where we freely identify
covectors and vectors by the metric. Lack of transversality occurs
when $\eta^h $ is tangent to $T_{(y, \eta)} (S^*_N M)$. The latter
is the kernel of $d y_n$. But $d y_m (\eta^h) = d y_m (\eta)= 0 $
if and only if $\eta \in T N$. 

It follows from Lemma \ref{NSYMP} that the symplectic volume form of $S^*_N M \backslash S^* N$
is $\omega^{n-1} |_{S_N^* M}$. The following Lemma gives a useful alternative formula:

\begin{lem} \label{dmuLN} 
Define  $$d\mu_{L, N} =  \iota_{\Xi} d\mu_L \;|_{S^*_N M}, $$
where  as above,   $d\mu_L$ is Liouville measure on  $S^* M$.  Then $$d \mu_{L, N}= \omega^{m-1} |_{S_N^* M}. $$
\end{lem}

Indeed,  $d \mu_L = \omega^{m-1} \wedge \alpha$, and $ \iota_{\Xi} d\mu_L = \omega^{m-1}$.  

\begin{cor} \label{COR} $\hcal^{m-1} (N) = \frac{1}{\beta_m} \int_{S^*_N M} |\omega^{m-1}|$. \end{cor}




\subsection{Hausdorff measure and Crofton formula for real geodesic arcs}

   First we sketch a proof of the integral geometry estimate using geodesic arcs rather than local coordinate
line segments. For background on integral geometry and  Crofton type formulae we refer to \cite{AB,AP}. As explained there, a Crofton
formula arises from a double fibration

$$\begin{array}{lllll} && \ical  && \\ &&&&\\
& \pi_1 \;\swarrow & & \searrow \;\pi_2 & \\ &&&& \\ \Gamma  &&&&
B,
\end{array}$$
where $\Gamma$ parametrizes a family of submanifolds $B_{\gamma}$ of $B$. The points  $b \in B$
then parametrize a family of submanifolds $\Gamma_b = \{\gamma \in \Gamma: b \in B_{\gamma}\}$
and the top space is the incidence relation in $B \times \Gamma$ that $b \in B_{\gamma}.$

We would like to define $\Gamma$ as the space of geodesics of $(M, g)$, i.e.
the space of orbits of the geodesic flow on $S^* M$.  Heuristically,
the space of geodesics is the quotient space $S^* M/\R$ where $\R$ acts by the geodesic flow $G^t$  (i.e. the Hamiltonian flow of
 $H$). Of course,
for a general (i.e. non-Zoll) $(M, g)$ the `space of geodesics' is not a Hausdorff space and so we do not have a simple analogue
of the space of lines in $\R^n$. Instead we consider  the space $\gcal_T$ of geodesic arcs of length $T$. 
 If we only use partial orbits of length $T$, no two partial orbits are equivalent and
 the space of geodesic arcs $\gamma_{x, \xi}^T$ of length $T$  is simply
parametrized by $S^* M$.  Hence we let  $B = S^* M$ and also $\gcal_T \simeq S^* M$.   The fact that different arcs of length $T$ of the same geodesic are distinguished leads to some
redundancy.

In the following, let $L_1$ denote the length of the shortest
closed geodesic of $(M, g)$.

\begin{prop}\label{CROFTONEST}  Let $N \subset M$ be any smooth hypersurface\footnote{The same formula is true
if $N$ has a singular set $\Sigma$ with $\hcal^{m-2}(\Sigma) < \infty$}, and let $S^*_N M$ denote
the unit covers to $M$ with footpoint on $N$. Then for
$0 < T < L_1,$
$$\hcal^{m-1}(N)  = \frac{ 1}{\beta_mT}  
 \int_{S^* M}
\# \{t \in [- T, T]: G^t(x, \omega) \in S^*_N M\} d\mu_L(x,
\omega),$$
where $\beta_m $ is $2 (m-1)!$ times  the volume of the unit ball in $\R^{m-2}$.
\end{prop}

\begin{proof}

By Corollary \ref{COR},   the Hausdorff measure of $N$ is given by 
\begin{equation} \label{HNN}\begin{array}{lll} \hcal^{m-1}(N) & = & \frac{1}{\beta_m}
\int_{S^*_N M} |\omega^{m-1}|. \end{array} \end{equation}

 We  use  the  Lagrange (or more accurately, Legendre) immersion,
 $$\iota: S^* M \times \R \to S^*M \times S^* M, \;\; \iota(x,
 \omega, t) = (x, \omega, G^t(x, \omega)), $$
 where as above,  $G^t$ is the geodesic flow \eqref{gtdef}.
 We also let $\pi: T^* M \to M$ be the standard projection.
We restrict $\iota$ to $S^* M \times [-T, T]$ and define the
incidence relation $$ \ical_T = \{((y, \eta),  (x, \omega), t)
\subset S^*M \times S^*M \times [- T, T]:  (y, \eta) =   G^t(x,
\omega)\}, $$  which is isomorphic to $[- T, T ] \times S^*M$
under $\iota$. We  form the diagram
$$\begin{array}{lllll} && \ical_T \simeq S^* M \times [-T, T]  && \\ &&&&\\
& \pi_1 \;\swarrow & & \searrow \;\pi_2 & \\ &&&& \\ S^* M \simeq \gcal_T &&&&
S^* M,
\end{array}$$
using the two natural projections, which in the local parametrization take the form
$$\pi_1(t, x, \xi) = G^t(x, \xi), \;\;\; \pi_2(t, x, \xi) = (x, \xi). $$ As noted above, the bottom
left $S^*M$ should be thought of as the space of geodesic arcs.
The fiber $$\pi_1^{-1}(y, \eta) = \{(t, x, \xi) \in [-T, T ] \times S^* M: G^t(x, \xi)
= (y, \eta)\} \simeq \gamma_{(y, \eta)}^T$$ may be identified with  the geodesic segment through
$(y, \eta)$ and the fiber $\pi_2^{-1} (x, \omega) \simeq [- T,
T]$.

We  `restrict' the  diagram above to $S^*_N M$:  
\begin{equation} \label{DIAGRAM} \begin{array}{lllll} && \ical_T \simeq S_N^* M \times [-T, T]  && \\ &&&&\\
& \pi_1 \;\swarrow & & \searrow \;\pi_2 & \\ &&&& \\ (S^*_N M)_T &&&&
S_N^* M,
\end{array} \end{equation}
where 
$$(S^*_N M)_{T} = \pi_1 \pi_2^{-1} (S_N^* M) =  \bigcup_{|t| < T} G^t(S^*_N M).$$

 We  define the Crofton density $\phi_T$ on $S_N^* M$ corresponding to   the diagram \eqref{DIAGRAM} \cite{AP} (section 4) by 
\begin{equation} \label{CROFDEN} \phi_T = (\pi_2)_* \pi_1^*
d\mu_L. \end{equation} Since the fibers of $\pi_2$ are 1-dimensional, $\phi_T$
is a differential form of dimension $2 \dim M - 2$ on $S^*M$.  To make it smoother, we
can introduce a smooth cutoff $\chi$  to $(-1,1)$,  equal to $1$ on $(- \half, \half)$, and use
$\chi_T(t) = \chi(\frac{t}{T}). $  Then $\pi_1^* (d\mu_L \otimes
\chi_T dt)$ is a smooth density on $\ical_T$.

\begin{lem} \label{phiT} The Crofton density \eqref{CROFDEN} is given by, $\phi_T = T d\mu_{L, N} $ \end{lem}

\begin{proof}

In \eqref{DIAGRAM} we defined the map   
$\pi_1: (y, \eta, t) \in S^*_N M \times [-T,T]  \to G^t(y, \eta) \in (S^* M)_{\epsilon}$. We first claim that
$\pi_1^* d\mu_L = d\mu_{L, N} \otimes dt. $ This is essentially the same as Lemma \ref{dmuLN}. Indeed,
$d \pi_1 (\frac{\partial}{\partial t} )= \Xi$, hence  $\iota_{\frac{\partial}{\partial t}} \pi_1^* d\mu_L |_{(t, y, \eta)}
= (G^t)^* \omega^{m-1} = \omega^{m-1}   |_{T_{y, \eta} S^*_N M}$. 




 Combining Lemma \ref{phiT} with \eqref{HNN} gives
\begin{equation}  \label{HDPHIT} \int_{S^*_N M} \phi_T = \int_{\pi_2^{-1} (S^*_N M)} d\mu_L =  T\beta_m\hcal^{m-1}(N).  \end{equation}  

\end{proof}

We then relate the integral on the left side to numbers of intersections of geodesic arcs with $N$. 
The relation is given by  the co-area formula:  if $f: X \to Y$ is a
smooth map of manifolds of the same dimension and if $\Phi$ is a
smooth density on $Y$, and if $\# \{f^{-1}(y)\} < \infty$ for
every regular value $y$, then
$$ \int_X f^* \Phi = \int_Y \# \{f^{-1}(y)\}\; \Phi. $$
If we set set $X = \pi_2^{-1}(S^*_N M), \; Y = S^* M, $ and $f =
\pi_1|_{\pi_2^{-1}(S^*_N M)}$ then   the co-area formula gives,
\begin{equation} \label{COAREA} \int_{\pi_2^{-1}(S^*_N M)} \pi_1^* d\mu_L = \int_{S^* M}
\# \{t \in [- T, T]: G^t(x, \omega) \in S^*_N M\} d\mu_L(x,
\omega). \end{equation} 

Combining \eqref{HDPHIT} and \eqref{COAREA} gives the  result stated in Proposition \ref{CROFTONEST},

\begin{equation} \label{CONCLUSION}   T \beta_m \hcal^{m-1}(N) = 
 \int_{S^* M}
\# \{t \in [- T, T]: G^t(x, \omega) \in S^*_N M\} d\mu_L(x,
\omega). \end{equation} 

\end{proof}

\subsection{Proof of Lemma \ref{DF2}}

The next step is to complexify.

\begin{proof}

We complexify the Lagrange immersion $\iota$ from a line (segment)
to a strip in $\C$: Define
$$F: S_{\epsilon} \times S^*M \to M_{\C}, \;\;\; F(t + i
\tau, x, v) = \exp_x (t + i \tau) v, \;\;\; (|\tau| \leq \epsilon)
$$ By definition of the Grauert tube, $\psi$ is surjective onto
$M_{\epsilon}$.  For each $(x, v) \in S^* M$,
$$F_{x, v}(t + i \tau) = \exp_x (t + i \tau) v $$
is a holomorphic strip. Here, $S_{\epsilon} = \{t + i \tau \in \C:
|\tau| \leq \epsilon\}. $ We also denote by
 $S_{\epsilon, L} = \{t + i \tau \in \C:
|\tau| \leq \epsilon, |t| \leq L \}. $

  Since $F_{x, v}$ is a holomorphic strip,
$$F_{x, v}^*(\frac{1}{\lambda} dd^c \log |\psi_j^{\C}|^2) =
 \frac{1}{\lambda} dd^c_{t + i \tau} \log |\psi_j^{\C}|^2 (\exp_x (t + i \tau)
 v) = \frac{1}{\lambda} \sum_{t + i \tau: \psi_j^{\C}(\exp_x (t + i
 \tau) v) = 0} \delta_{t + i \tau}.
$$ Put:
\begin{equation} \label{acal} \acal_{L, \epsilon}  (\frac{1}{\lambda} dd^c \log |\psi_j^{\C}|^2) = \frac{1}{\lambda} \int_{S^* M} \int_{S_{\epsilon, L}}
dd^c_{t + i \tau} \log |\psi_j^{\C}|^2 (\exp_x (t + i \tau) v)
d\mu_L(x, v). \end{equation}   A key observation of \cite{DF,Lin} is that
\begin{equation} \label{MORE}
\#\{\ncal_{\lambda}^{\C} \cap F_{x,v}(S_{\epsilon, L}) \} \geq \#\{\ncal_{\lambda}^{\R} \cap F_{x,v}(S_{0, L}) \}, 
\end{equation}
since every real zero is a complex zero. 
It follows then from    Proposition \ref{CROFTONEST} (with $N = \ncal_{\lambda}$) that
$$\begin{array}{lll} \acal_{L, \epsilon}  (\frac{1}{\lambda} dd^c \log
|\psi_j^{\C}|^2) &= &  \frac{1}{\lambda} \int_{S^* M}
\#\{\ncal_{\lambda}^{\C} \cap F_{x,v}(S_{\epsilon, L}) \} d
\mu(x,v) \\ && \\
&\geq & \frac{1}{\lambda} \hcal^{m-1}(\ncal_{\psi_{\lambda}}).\end{array} $$ 
Hence to obtain an upper bound on $\frac{1}{\lambda} \hcal^{m-1}(\ncal_{\psi_{\lambda}})$ it suffices
to prove that there exists $M < \infty$ so that
\begin{equation} \label{acalest} \acal_{L, \epsilon}  (\frac{1}{\lambda} dd^c \log
|\psi_j^{\C}|^2)  \leq M. \end{equation}

To prove \eqref{acalest}, we observe that since  $dd^c_{t + i \tau} \log |\psi_j^{\C}|^2 (\exp_x (t + i \tau)
v)$ is a positive $(1,1)$ form on the strip, the integral over
$S_{\epsilon}$ is only increased if we integrate against a
positive smooth test function $\chi_{\epsilon} \in
C_c^{\infty}(\C)$ which equals one on $S_{\epsilon, L}$ and vanishes
off $S_{2 \epsilon, L} $. Integrating  by
parts the $dd^c$ onto $\chi_{\epsilon}$, we have 
$$\begin{array}{lll} \acal_{L, \epsilon} (\frac{1}{\lambda} dd^c \log |\psi_j^{\C}|^2) &\leq &  \frac{1}{\lambda} \int_{S^* M} \int_{\C}
dd^c_{t + i \tau} \log |\psi_j^{\C}|^2 (\exp_x (t + i \tau) v)
\chi_{\epsilon} (t + i \tau) d\mu_L(x, v) \\ && \\  &= & \frac{1}{\lambda} \int_{S^* M} \int_{\C}
 \log |\psi_j^{\C}|^2 (\exp_x (t + i \tau) v)
dd^c_{t + i \tau} \chi_{\epsilon} (t + i \tau) d\mu_L(x, v) .
\end{array}$$

Now write $\log |x| = \log_+ |x| - \log_- |x|$. Here $\log_+ |x| = \max\{0, \log |x|\}$ and
$\log_ |x|= \max\{0, - \log |x| \}. $ Then we need upper bounds for 
$$  \frac{1}{\lambda} \int_{S^* M} \int_{\C}
 \log_{\pm} |\psi_j^{\C}|^2 (\exp_x (t + i \tau) v)
dd^c_{t + i \tau} \chi_{\epsilon} (t + i \tau) d\mu_L(x, v) .$$
For $\log_+$ the upper bound is an immediate consequence of Proposition \ref{PW}.  For $\log_-$ the
bound is subtler: we need to show that $|\phi_{\lambda}(z)| $ cannot be too small on too large a set. 
As we know from Gaussian beams, it is possible that $|\phi_{\lambda}(x) | \leq C e^{- \delta \lambda} $
on sets of almost full  measure in the real domain;
we  need to show that nothing worse can happen. 

 The map \eqref{E} is a diffeomorphism and since $B_{\epsilon}^* M = \bigcup_{0 \leq \tau \leq \epsilon} S^*_{\tau} M$
we also have that 
$$E:   S_{\epsilon, L} \times S^* M  \to M_{\tau}, \;\;\; E(t + i \tau, x, v) = \exp_x (t + i \tau) v  $$
is  a diffeomorphism for each fixed $t$. Hence  by letting $t$ vary, $E$ 
is a smooth fibration with  fibers given by  geodesic arcs.  Over a point $\zeta \in M_{\tau}$ the fiber of the map is a geodesic arc
$$\{ (t + i \tau, x, v): \exp_x (t + i \tau) v = \zeta, \;\; \tau = \sqrt{\rho}(\zeta)\}. $$ Pushing forward the measure $
dd^c_{t + i \tau} \chi_{\epsilon} (t + i \tau) d\mu_L(x, v) $ under $E$ gives  a positive measure $d\mu$ on $M_{\tau}$. 
We claim that 
\begin{equation}\label{PUSH}  \mu: = E_* \; dd^c_{t + i \tau} \chi_{\epsilon} (t + i \tau) d\mu_L(x, v)  =\left (\int_{\gamma_{x, v}} 
\Delta_{t + i \tau} \chi_{\epsilon} ds \right) dV_{\omega}, \end{equation}
where $dV_{\omega}$ is the  K\"ahler volume form $\frac{\omega^m}{m!} $ (see \S \ref{AC}.)

In fact, $d\mu_{L}$ is equivalent under $E$ to the contact volume form $\alpha \wedge \omega_{\rho}^{m-1}$
where $\alpha = d^c \sqrt{\rho}$.
Hence the claim amounts to saying that the K\"ahler volume form is $d \tau$ times the contact volume form.
In particular  it is a smooth  (and  of course signed)  multiple $J$  of the K\"ahler volume form $dV_{\omega}$, and we
do not need to know the coefficient function $J$ beyond that it is bounded above and below by constants independent of
$\lambda$.
We then have
\begin{equation} \label{JEN}  \int_{S^* M} \int_{\C}
 \log |\psi_j^{\C}|^2 (\exp_x (t + i \tau) v)
dd^c_{t + i \tau} \chi_{\epsilon} (t + i \tau) d\mu_L(x, v) = \int_{M_{\tau}} 
\log |\psi_j^{\C}|^2   J d V.  \end{equation}
To complete the proof of \eqref{acalest} it suffices  to prove that the right side is $\geq - C \lambda$ for some $ C> 0$.

We   use  the well-known

\begin{lem} \label{HARTOGS}  (Hartog's Lemma;  (see
\cite[Theorem~4.1.9]{HoI-IV}): Let $\{v_j\}$ be a sequence of subharmonic functions in an
open set $X \subset \R^m$ which have a uniform upper bound on any
compact set. Then either $v_j \to -\infty$ uniformly on every
compact set, or else there exists a subsequence $v_{j_k}$ which is
convergent to some  $u \in L^1_{loc}(X)$. Further,  $\limsup_n
u_n(x) \leq u(x)$ with equality almost everywhere. For every
compact subset $K \subset X$ and every continuous function $f$,
$$\limsup_{n \to \infty} \sup_K (u_n - f) \leq \sup_K (u - f). $$
In particular, if $f \geq u$ and $\epsilon > 0$, then $u_n \leq f
+ \epsilon$ on $K$ for $n$ large enough. \end{lem}

 This Lemma implies the desired lower bound on \eqref{JEN}: 
  there exists $C > 0$ so that  \begin{equation}
\label{LOGINT} \frac{1}{\lambda} \int_{M_{\tau} } \log |\psi_{\lambda}| J d V \geq - C.  \end{equation}
For 
if not, there exists a  subsequence of eigenvalues $\lambda_{j_k}$
so that $\frac{1}{\lambda_{j_k}}\int_{M_{\tau}} \log |\psi_{\lambda_{j_k}}| J d V \to - \infty. $  By Proposition \ref{PW}, $\{\frac{1}{\lambda_{j_k}} \log |\psi_{\lambda_{j_k}}|\}$   has a uniform upper bound.
Moreover  the sequence does not tend uniformly to $-\infty$  since  $||\psi_{\lambda}||_{L^2(M)} = 1$. 
It follows that a further subsequence tends in $L^1$ to a
limit $u$ and by the dominated convergence theorem the limit of \eqref{LOGINT} along the sequence
equals $\int_{M_{\tau}} u J dV \not= - \infty.$ This contradiction concludes the proof of \eqref{LOGINT}, hence
 \eqref{acalest}, and thus 
 the theorem.


\end{proof}

\end{document}